\documentclass[12pt]{article}
\usepackage{amsmath,amssymb,amsfonts,amsthm,graphicx,epsfig}
\usepackage[usenames,dvipsnames]{color}
\usepackage{hyperref}
\hypersetup{
    colorlinks=false, %set true if you want colored links
    linktoc=all,     %set to all if you want both sections and subsections linked
    linkcolor=blue,  %choose some color if you want links to stand out
    linktocpage,
}

\newcommand{\old}[1]{}
\newcommand{\eps}{\varepsilon}

\newcommand{\yellow}[1]{}

\newcommand{\R}{{\mathbb R}}

\renewcommand{\R}{\mathbb{R}}

\newtheorem{theorem}{Theorem}

\newtheorem{lemma}[theorem]{Lemma}
\newtheorem{proposition}[theorem]{Proposition}
\newtheorem{corollary}[theorem]{Corollary}

\begin{document}
\date{}

\title
{Sets that Support a Joint Distribution}
\author{
Christopher Coscia\thanks{Department of Mathematics, Tufts University,
Medford, MA 02155. E-mail: {\tt christopher.coscia@tufts.edu}}
\and Martin Tassy
\and Peter Winkler\thanks{Department of Mathematics, Dartmouth
  College, Hanover, New Hampshire 03755. E-mail: {\tt peter.winkler@dartmouth.edu}}
}

\maketitle

\begin{abstract}
Given probability distributions $\mu$ and $\nu$ on measure spaces $X$ and $Y$, and a closed set
$S \subseteq X \times Y$, when is there a probability distribution on
$X \times Y$ whose marginals are $\mu$ and $\nu$, and whose support is precisely $S$?

We answer the question when the marginals are discrete, and when the
marginals are continuous distributions on the real line.  Of special interest is the case where
$S \subseteq [0,1]^2$ and $\mu$ and $\nu$ are Lebesgue measure; then the above question
is tantamount to ``when is $S$ the support of a doubly stochastic measure?".

The discrete case is generalized to determine when a (possibly infinite)
edge-capacitated, node-weighted graph supports a full, nowhere-zero flow;
for the continuous case we provide a particularly straightforward characterization
when the set in question is regular (i.e., is the closure of its interior).
\end{abstract}

\section{Introduction}
We address the fundamental question of what sets of values are possible when two random
variables with given distributions interact.  This issue has presumably arisen before,
but was re-inspired among the authors by recent interest in permutons, about which there
is more below.

Suppose we are given two random variables, one with probability distribution $\mu$ on
measure space $X$, the other with distribution $\nu$ on $Y$ (which we always take to be
disjoint from $X$).  Their joint distribution will be supported by some closed set
$S \subseteq X \times Y$; what sets $S$ are possible?

When $X$ and $Y$ are (finite or countably infinite) discrete sets, the problem can be thought
of in the following graph-theoretic terms.  A bipartite graph $G = \langle X \cup Y,E \rangle$
is given along with positive node weights that sum to 1 on each bipart. When is
there a ``full nowhere-zero flow" from $X$ to $Y$ in $G$, that is, an
assignment of positive weights to the edges of $G$ such that the sum of the weights of
the edges incident to a node is equal to the weight of the node?

The finite case has a matrix interpretation as well.  An $m \times n$ $\{0,1\}$ ``template"
matrix $M$ is provided along with real vectors $\vec{r} = (r_1,\dots,r_m)$ and
$\vec{c} = (c_1,\dots,c_n)$.  The question now is whether there is an $m \times n$
nonnegative-real-valued matrix $A$ with row sums $\vec{r}$ and column sums $\vec{c}$,
whose entries are positive exactly where there is a 1 in $M$.

In the continuous case, we limit ourselves to distributions $\mu$ and $\nu$ on the
real line $\R$ which are absolutely continuous with respect to Lebesgue measure. Then,
given a closed set $S \subseteq \R \times \R$, we ask whether there is a Borel probability
distribution on $\R \times \R$ supported by $S$ whose marginals are $\mu$ and $\nu$.
In such cases we can (and usually will) simplify by defining continuous increasing maps
$f$ and $g$ from $\R$ to $[0,1]$ such that $\mu(A) =
\lambda(f[A])$ and $\nu(B) = \lambda(g[B])$ for $A,B \subseteq \R$.  Then, replacing
$S$ by $(f \times g)[S] \subseteq [0,1] \times [0,1]$, we reduce the problem to the case
where $\mu$ and $\nu$ are both Lebesgue distributions $\lambda$ on unit intervals and
$S$ is a closed subset of a unit square.  (This reduction is known among statisticians
as Sklar's Theorem \cite{Sklar}.)

Borel measures on the unit square with uniform marginals are familiar objects in many
communities, and have been known at times as ``doubly stochastic measures" or two-dimensional
``copulas."  More recently, in \cite{HKMS}, they entered combinatorics as limit objects for
permutations, called {\em permutons.} Permutons enjoy a powerful variational principle
\cite{Tr,Mu,KKRW,BDMW} that can provide an asymptotic description of random members of
a large set of permutations.  The space $\mathcal{P}$ of permutons is compact in the
topology induced by the $L^\infty$ metric on cumulative distribution functions; equivalently,
a sequence of permutons converges if the mass they assign to any fixed Lebesgue-measurable set in
$[0,1]^2$ converges.

Permutons may be nonsingular (that is, absolutely continuous with respect to Lebesgue
measure).  The other extreme is a ``totally singular" permuton whose
support has Lebesgue measure 0.  Both types, and mixtures of the two, arise in the study
of large permutations.  Fig.~\ref{fig:test-sets} below shows eight subsets of the unit
square; which can support a permuton?

\begin{figure}[ht!]
        \centering
        \includegraphics[width=.8\linewidth,trim=0cm 9cm 0cm 6cm, clip=true]{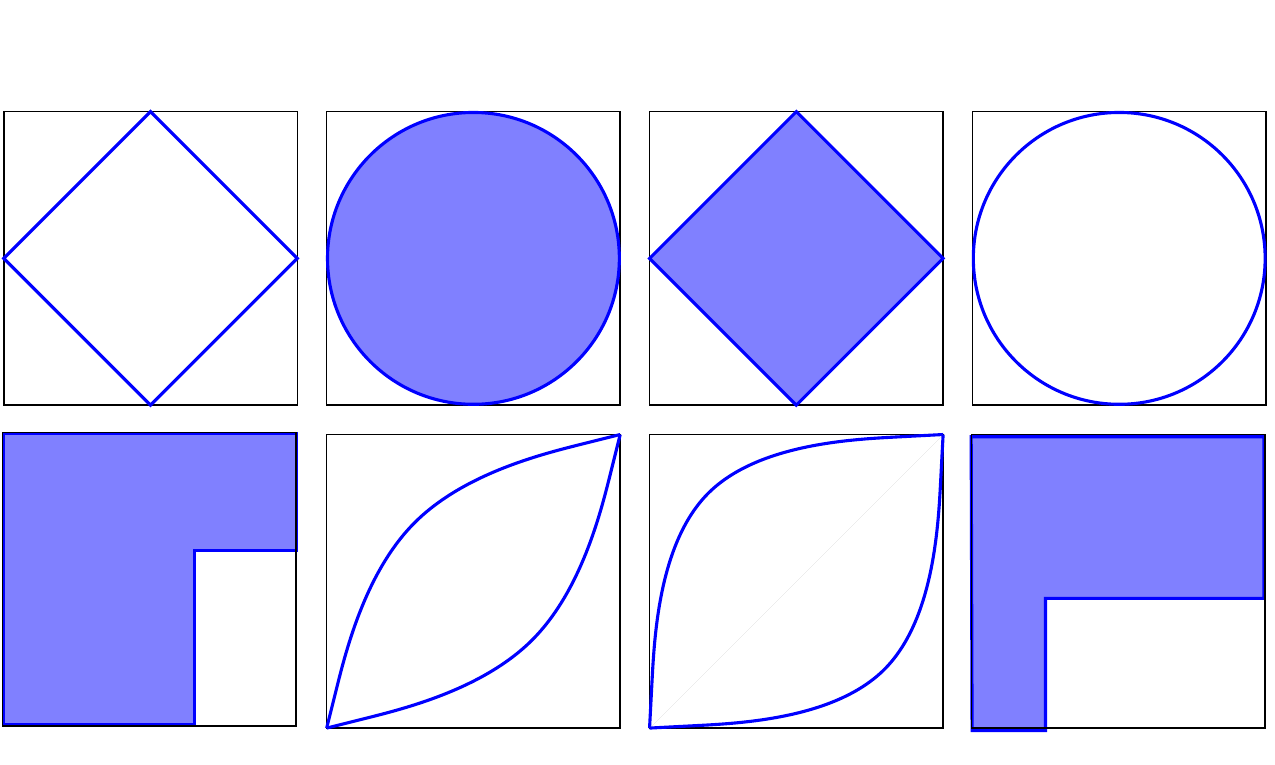}
        \caption{Which of these subsets of $[0,1]^2$ are the support of a probability
measure with uniform marginals?}
        \label{fig:test-sets}
\end{figure}

\section{Feasibility and Compliance}

We assume throughout this section that $X$, $Y$, $\mu$, $\nu$, and $S \subseteq X \times Y$
are given and fixed.  Together these data constitute a {\em system}, which we term
{\em feasible} if there is a distribution $\gamma$ on $X \times Y$ with $X$-marginal
$\mu$ and $Y$-marginal $\nu$, whose support is contained in $S$; and {\em solvable}
if there is one whose support is exactly $S$.  We will always assume either that
$\mu$ and $\nu$ are discrete, or that they are continuous distributions
on $\R$---in which case $S$ is assumed to be a closed set.

With the aim of simplifying expressions when possible, we will use absolute value signs
to denote the measures $\mu$ and $\nu$, both for subsets of $X$ or $Y$ and (in the discrete
case) individual elements.  Thus, for example, if $x \in U \subseteq X$, then $|x| := 
\mu(\{x\})$ and $|U| := \mu(U)$.

Given $U \subseteq X$ or $V \subseteq Y$, we denote by $N(U)$ (the ``neighborhood" of $U$,
a term borrowed from the graph-theoretic interpretation) the set of all $y \in Y$ such
that $(x,y) \in S$, and similarly for $N(V) \subseteq X$.

We begin with the (easier) question of feasibility, the answer to which, at least in
the discrete case, is well known.  In all cases an obvious necessary condition for
feasibility is that for all $U \subseteq X$, $|N(U)| \ge |U|$. 
(Otherwise, $|U| = \gamma((U \times Y) \cap S) > \gamma((X \times N(U)) \cap S) = |N(U)|$,
but $(U \times Y) \cap S \subseteq (X \times N(U)) \cap S$.)  This is sometimes
called the (weighted) ``Hall condition," since, in Philip Hall's famous marriage theorem,
it guarantees a matching covering all of $X$ in the uniform discrete case.  We don't
actually need a matching here, only (in the discrete case) a full flow.  In all cases,
discrete or continuous, the Hall condition is sufficient as well as necessary to get a joint
distribution whose support is a {\em subset} of $S$.

\begin{theorem}\label{thm:feasibility}
A finite system $\langle X,Y,\mu,\nu,S \rangle$ is feasible if and only if for every
measurable $U \subset X$, $|N(U)| \ge |U|$.
\end{theorem}

\begin{proof}
We will say that $U \subset X$ is {\em compliant} if $|N(U)| \ge |U|$,
and that a system is compliant if every measurable subset of $X$ is compliant.
We have already noted that compliance is a necessary condition for feasibility.

To prove sufficiency let us first observe that although it might not seem so, compliance
is symmetric in $X$ and $Y$.  The reason is simply that $N(Y \setminus N(U)) = X \setminus U$,
thus $|N(U)| - |U| = |N(Y \setminus N(U))| - |Y \setminus N(U)|$, so $U$ is compliant iff
$Y \setminus N(U)$ is.  (We are relying here on $|X|=|Y|$, but not on $|X|=1$; we don't
need $\mu$ and $\nu$ to be probability measures, just finite measures with the same total
mass.)

The finite case of sufficiency is a familiar consequence of the max-flow min-cut theorem
(see, e.g., Chv\'atal \cite{C}).  Let $D$ be the directed graph described as follows.
The vertices of $D$ consist of the elements of $X$ and of $Y$ together with a source
$s$ and a sink $t$.  Exiting $s$ is an arc to each $x \in X$ of capacity $|x|$; exiting
each $y \in Y$ is an arc to $t$ of capacity $|y|$.  Each pair $(x,y) \in S$ produces
an arc from $x$ to $y$ of unlimited capacity.

Fig.~\ref{fig:flow} illustrates a network corresponding to the system whose matrix
formulation stipulates row sums $(.1,.3,.2,.1,.3)$, column sums $(.2,.4,.3,.1)$, and
$5 \times 4$ template matrix
$$
\begin{pmatrix}
1 & 1 & 0 & 0 \\
1 & 0 & 0 & 0 \\
0 & 1 & 0 & 0 \\
1 & 0 & 1 & 1 \\
0 & 1 & 0 & 1 
\end{pmatrix}
$$

The cut shown corresponds to the Hall-tight set $\{1,2,3\} \subset X$.  Since it is a
minimum cut and has capacity $1 = |X|$, there is indeed a full flow (but not one that
uses all the arcs; the two edges that ``jump" the cut in the wrong direction will
carry no mass in a full flow).

\begin{figure}[ht!]
        \centering
        \includegraphics[width=.8\linewidth,trim=0cm 7cm 0cm 5cm, clip=true]{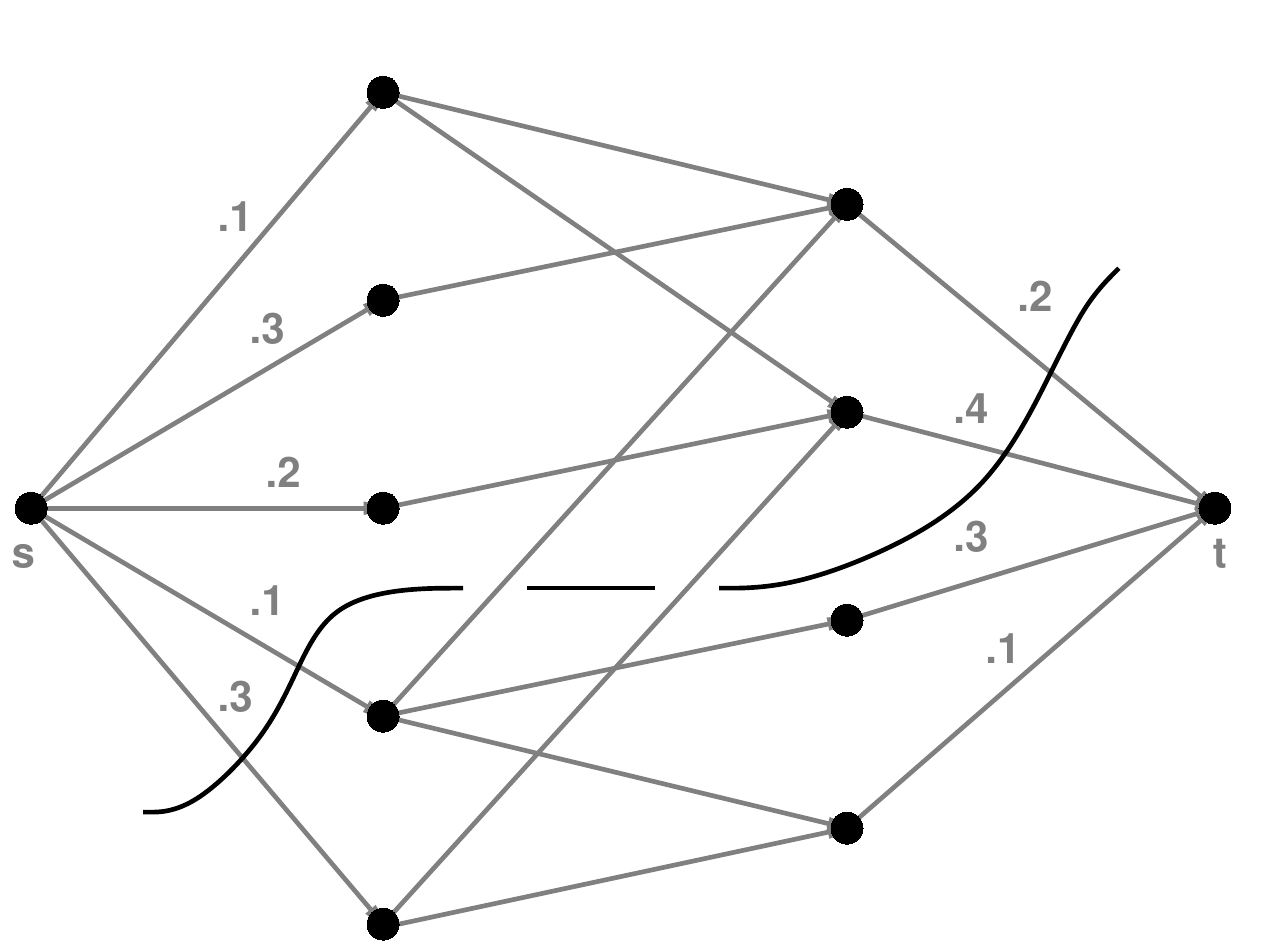}
        \caption{The network corresponding to a system, and a minimum cut.
The unlabeled edges have unlimited capacity.}
        \label{fig:flow}
\end{figure}

A flow of magnitude 1 on $D$ is what's needed, and if it is not
present there is a cut of capacity less than 1.  No edge of $D$ can have been cut;
thus, if $U$ is the set of $X$-nodes on the source side of the cut, all edges from
the source to $X \setminus U$ and from $N(U)$ to the sink must have been cut.  Then
$$
|N(U)| - |U| = |N(U)| + |X \setminus U| - 1 < 0
$$
so $U$ is not compliant.
\end{proof}

\medskip

The infinite discrete case will follow from the continuous case, which we proceed to now.
We begin with a continuous analog of (a special case of) the max-flow min-cut theorem.
Henceforth, all subsets to which we refer are assumed to be Lebesgue measurable.

\begin{theorem}\label{thm:cont-mfmc}
Let $\alpha$ and $\beta$, respectively, be distributions with the same total mass
on copies $X$ and $Y$ of the unit interval.  Suppose $\alpha$ and $\beta$ are each
bounded by Lebesgue measure, and let $S$ be a closed subset of the unit square $X \times Y$.
Let $r := \inf_{W \subseteq X} r(W)$ where $r(W) := \beta(N(W)) + \alpha(X \setminus W)$.
Then there is a distribution $\mu$ supported in $S$ with total mass $m$ and marginals
bounded by $\alpha$ and $\beta$, if and only if $0 \le m \le r$.
\end{theorem}

\begin{proof} Recall that a set function $f$ is {\em submodular} if $f(U \cup V) + f(U \cap V)
\le f(U) + f(V)$ for any sets $U$ and $V$, and {\em modular} if the quantities are equal.
The function $r(\cdot)$ is submodular since $|N(W)|$ is submodular and $|X \setminus W|$
is modular.\footnote{In fact, $r$ itself is submodular as a function of $X$; see, e.g.,
\cite{L2}, Theorem 2.5, where it is shown that the ``convolution" of a modular and a
submodular function is submodular.}  We use this observation to show that
the $\inf$ in the definition of $r$ can be replaced by $\min$, that is, there is always a
``realizer" $W$ for which $r = r(W)$.  To see this, let $W_1, W_2, \dots$ be chosen
so that $r(W_i)-r < 2^{-i}$.  Since $r(\cdot)$ is submodular and $r(W_i \cap W_j) \ge r$
by definition of $r$, we deduce that $r(\bigcup_{i>k} W_i) \le r + 2^{-k}$ and thus
$W := \bigcap_k (\bigcup_{i>k} W_i)$ will do the job.

For each $n>0$ let us cover $X \times Y$ with a $2^n \times 2^n$ dyadic square grid $G_n$;
cells that intersect $S$ are said to be {\em occupied}, else {\em empty}. Let $S_n$ be the
set of occupied cells.  The $2^n$ equal intervals into which $X$ has
been divided constitute a discrete set $X_n$ with measure induced by $\alpha$, and similarly for
$Y$, $Y_n$ and $\beta$.  $S_n$ induces a directed network $D_n$ on vertices
$\{s\} \cup X_n \cup Y_n \cup \{t\}$ as in the discrete case argument above.
A finite-capacity cut in this network can cut only edges incident to $s$ or $t$; if $W_n$
is the set of vertices of $X_n$ whose edges to $s$ are not cut, then the capacity of the cut is
$r_n := \beta(N(W_n)) + \alpha(X_n \setminus W_n)$.  Max-flow min-cut gives us a discrete distribution
of mass $r_n$ on the grid cells, with marginals bounded by $\alpha$ and $\beta$; that can be
converted to a continuous distribution $\mu_n$ on $\bigcup S_n$ with marginals bounded by
$\alpha$ and $\beta$, by taking the distribution on each cell to be independent with the
appropriate mass and marginals.

It follows from compactness that the sequence of distributions $\{\mu_n\}_n$ has a limit
measure $\mu$, still with marginals bounded by $\alpha$ and $\beta$, and total mass
$\ge \lim\inf r_n$.  But each $r_n \ge r$ since for each $n$, $\bigcup W_n$ satisfies the conditions
defining $r$. It follows that the total mass of $\mu$ is exactly $r$, and of course $\mu$
can be cut down by a constant factor to give any total mass between 0 and $r$.

It remains only to verify that $\mu$ is supported in $S$.  Since $S$ is closed, if the point
$P \not\in S$ then there is a dyadic square $Q = [(i{-}1)2^{-k},i2^{-k}] \times [(j{-}1)2^{-k},j2^{-k}]$
containing $P$ and contained within $X \times Y \setminus S$.  For $n \ge k$ all the dyadic
cells comprising $Q$ will be empty in $G_n$, thus $P$ is not in the support of $\mu$. \end{proof}

\begin{corollary}\label{cor:shaved}
If $\alpha$ and $\beta$ have the same total mass $c$ and $S$ is compliant (meaning, here,
that $\beta(N(U)) \ge \alpha(U)$ for any (measurable) $U \subseteq X$) then there is a
measure supported in $S$ of total mass $c$ whose marginals are exactly $\alpha$ and $\beta$.

In the case that $\alpha = \beta =$ Lebesgue measure, if $S \subseteq X \times Y$ is closed
and compliant, there is a permuton supported in $S$.
\end{corollary}

A measure supported in $S$ with given Lebesgue-bounded marginals, such as guaranteed
by the above corollary, will be called a ``subpermuton."

To get the infinite discrete case from Theorem~\ref{thm:cont-mfmc} or
Corollary ~\ref{cor:shaved}, we just divide up $[0,1]$ into intervals $I_x$ of length
$|x|$ for each $x \in X$, and another copy of $[0,1]$ similarly into $J_y$ for each $y \in Y$. 
Now the permuton case is applied with support set $\bigcup_{(x,y) \in S}[I_x \times J_y]$,
noting that the neighborhood of a set $U$ in the (horizontal) unit interval is the union of
all the $J_y$ for which $U$ merely intersects some $I_x$ with $(x,y) \in S$.  Thus the
compliance of sets that aren't unions of $I_x$'s follows from the compliance of those
that are.

\section{Solvability and Expansiveness}

We have seen that compliance equals feasibility.  Determining whether a system is solvable,
that is, whether $S$ actually is the support of some joint distribution, is a bit more subtle.
One apparent obstruction to solvability in the discrete case is the existence of a subset
$U \subset X$ which is ``tight'' for compliance, that is, $|N(U)| = |U|$, together with a pair
$(x,y) \in S$ with $y \in N(U)$ but $x \not\in U$.  Then all of the flow to $N(U)$ must
come from $U$, so nothing can flow from $x$ to $y$.

In general, we say that $U \subset X$ is {\em expansive} if either $|N(U)| > |U|$, or
$|N(U)| = |U| = |N(N(U))|$.  We call $S$ expansive if every measurable $U \subset X$
is expansive.  We observe immediately, for continuous as well as discrete systems,
that expansiveness is necessary for solvability.

Spoiler alert: Expansiveness is not sufficient for solvability in the general continuous case.

\subsection{The finite discrete case}

Our general approach for showing that expansiveness is sufficient for solvability in discrete
cases will be to show that, in the presence of expansiveness, for any edge $e = (x,y) \in S$
there is a full flow $\phi_e$ that is positive on $e$.  A convex combination of the flows
$\phi_e$, for all $e \in S$, will then be nowhere zero and still full, proving solvability.

To get $\phi_e$, we lower the weights of $e$'s endpoints by some $\eps > 0$, show that the
resulting system is still compliant, apply Theorem~\ref{thm:feasibility} to get a flow
$\phi$ of mass $1 - \eps$, then add flow $\eps$ to $\phi(e)$.

It will be helpful to give a name to the degree to which a set is compliant.  The {\em slack}
$s(U)$ of a subset of $X$ (or of $Y$) is defined as $|N(U)|-|U|$.  Thus a compliant system
is one in which all slacks are nonnegative; a system is expansive iff it is compliant and for
all $U \subset X$, $s(U) = 0$ implies $s(N(U)) = 0$.  

\begin{theorem}\label{thm:finite} A finite system is expansive iff it is solvable.
\end{theorem}

\begin{proof} Following the script above, let $e = (x,y) \in S$ in an expansive system.
Suppose $U \subset X$ with $y \in N(U)$ but $x \not\in U$.  Then
$x \in N(N(U)) \setminus U$, so $s(U) > 0$ by expansiveness.  Let $s$ be the minimum slack
among all such $U$, and let $0 < \eps < \min(|x|,|y|,s)$.  We define a new system with $S$
unchanged, and $\mu$ and $\nu$ as before except $|x|$ and $|y|$ have been lowered by
$\eps$ (thus the mass of $X$ and of $Y$ in the new system is only $1-\eps$).  Apply
Theorem~\ref{thm:feasibility} to get a flow of mass $1 - \eps$, and add back $\eps$
worth of flow to the edge $e$ to get a full flow $\phi_e$ in the original system
that is positive on $e$.  Put $\phi = \frac1{|S|}\sum_{e \in S}\phi_e$ to get a full
nowhere-zero flow as required.
\end{proof}

\subsection{Nonbipartite graphs and networks}

There are (at least) two ways to extend Theorem~\ref{thm:finite} to general
graphs.\footnote{Thanks to Richard Stanley for suggesting extension to nonbipartite graphs.}
For the next theorem, if $U$ is a set of vertices, $N(U)$ denotes the excluded
neighborhood $\{v \not\in U:~u \sim v$ for some $u \in U\}$.
A set of edge weights will be said to be {\em compatible} with a set of vertex
weights if the weight of each vertex is the sum of the weights of its incident edges.

\begin{theorem} \label{thm:nonbip1} Let $G$ be a finite undirected graph with
positive vertex weights $u \mapsto |u|$.  Then there is a compatible nonnegative edge
weighting $e \mapsto w(e)$ iff for every nonempty independent set $I \subset V(G)$,
$|N(I)| \ge |I|$; and there is a compatible positive edge weighting if, in addition, 
whenever $|N(I)| = |I|$, $N(N(I)) = I$.
\end{theorem}

\begin{proof} Necessity in each case is evident; we proceed to sufficiency for the
nonnegative edge weights, by induction on the cardinality $n$ of $V(G)$.
Defining slack $s(I) := |N(I)|-|I|$ as before, let $s$ be the minimum of $s(I)$ over all
nonempty proper independent $I \subset V(G)$.  If $s=0$, let $I$ be a
witness to that.  Then $I \cup N(I)$ induces a graph with compatible nonnegative
edge-weights, by induction, and so does $V(G) \setminus (I \cup N(I))$ (since an
independent $J$ in that subgraph with $|N(J)| < |J|$ would cause $I \cup J$ to fail
compliance).  Combining the edge-weightings provides one for all of $G$.

If $s > 0$, we choose an $I$ with minimum $s(I)$ and reduce vertex weights in $N(I)$
until $s(I)=0$, then proceed as above. (See \cite{DK} for another proof of the first
part of the theorem.)

Sufficiency for the second part of the theorem follows from the argument in the proof
of Theorem~\ref{thm:finite}.  \end{proof}

\medskip

\begin{theorem} Let $D$ be a finite directed network with minimum cuts (and thus
maximum flows) of value $c$.  Then $D$ has a {\em nowhere-zero} flow of value $c$
if and only if every arc belongs to a path from source to sink, and no arc crosses a
minimum cut in the reverse direction.
\end{theorem}

\begin{proof} Let $s$ be the source and $t$ the sink.  The conditions are clearly
necessary, since an arc that's not on an $s \to t$ route is unusable in any flow,
and one that crosses a minimum cut backwards can't carry any mass in a full flow.
To show sufficiency, we suppose the conditions hold, and employ a variation of
the argument in the proof of Theorem~\ref{thm:finite}.

Let $P$ be a path in $D$ from $s$ to $t$, so that
$$
P = s \to P_1 \to u_1 \to v_1 \to P_2 \to u_2 \to v_2 \to P_3 \to \cdots \to P_k \to t
$$
where the arcs $(u_i,v_i)$ are those that belong to some minimum cut.  Because no arcs
of $D$ cross a minimum cut backwards, each $(u_i,v_i)$ belongs to a different minimum cut.
Let $k$ be at least the number of times any arc appears in $P$, and also at least the
number of times any cut $C$ is crossed in the forward direction by $P$.  Pick $\eps > 0$
such that $k\eps$ is less than any arc capacity, and such that every nonminimum cut has
value greater than $c + k\eps$.

Now shave off $\eps$ from the capacity of every arc in $P$, to get a new network $D'$
whose minimum cut value is only $c-\eps$.  Let $f'_P$ be a full flow (of mass $c-\eps$)
on $D'$, and add back $\eps P$ to get a full $c$-flow $f_P$ on the original network $D$.
The average of $f_P$ over all paths $P$ from $s$ to $t$ is a full flow on $D$,
and since every arc is on some $P$, we have the nowhere-zero flow as required. \end{proof}

\subsection{The infinite, capacitated, discrete case}

The sufficiency of expansiveness in the (countably) infinite discrete case will follow,
as in the feasibility case, from a later theorem about permutons. Thus, we can't really
justify presenting a proof of this case here.  As a compromise, since flow problems often
come with edge-capacities---and, indeed, capacity arises on its own in the continuous case---we
instead prove a capacitated version of the infinite discrete case.

In the previous section, an edge $e = (x,y)$ was allowed to carry as much flow as we wanted;
thus, anything up to $\min(|x|,|y|)$.  Suppose we introduce capacities $c(e) = c(x,y)$ and
permit only flows $\phi$ for which $\phi(e) \le c(e)$ for each edge $e$.

Clearly, Hall's condition must be amended; if some edges from $U \subset X$ have
low capacity, the flow from $U$ to $N(U)$ may be limited to much less than $|N(U)|$.
In particular, if the sum of the capacities of the edges from $U$ to $y \in N(U)$ is
less than $|y|$, then effectively $|y|$ is replaced by that sum when we add weights to
compute what was the $|N(U)|$ term in the Hall condition.

Accordingly, let us define $|y|_U$ to be the minimum of $|y|$ and the sum $c(U,y) :=
\sum_{x \in U} c(x,y)$ of the capacities of the edges from $U$ to $y$,
and $|N(U)|_U$ to be $\sum\{|y|_U:~y \in N(U)\}$.  Then, in order
to have a flow, every $U \subset X$ must satisfy the ``capped" Hall condition
$|N(U)|_U \ge |U|$.  Sufficiency follows easily from the max-flow min-cut theorem.

\begin{proposition}\label{capped-hall} Let $G = \langle X,Y,E \rangle$ be a (possibly
infinite) node-weighted, edge-capped bipartite graph with $|X|=|Y|=1$.  Then $G$ supports
a unit flow from $X$ to $Y$ if and only if for every $U \subseteq X$, $|N(U)|_U \ge |U|$.
\end{proposition}

\begin{proof} As in the proof of Theorem~\ref{thm:feasibility}, we augment $G$ with a source $s$
connected to each $x \in X$ by an edge of capacity $|x|$, and a sink $t$ connected to each
$y \in Y$ by an edge of capacity $|y|$.  Edges between $X$ and $Y$ are given the capacities
assigned to them in $G$.

If there is no unit flow from $s$ to $t$ in this network, there must be a cut of capacity
less than 1; let $U$ contain the elements of $X$ on the source side of the cut.  Then
for every $y \in N(U)$, either the edge from $U$ to $y$ or the edge from $y$ to $t$ must
be cut, thus each such $y$ is responsible for contributing at least $|y|_U$ to the magnitude
of the cut.  Since the edges from $s$ to $X \setminus U$, whose total capacity is $1-|U|$,
must also be cut, we have $1 - |U| + |N(U)|_U < 1$, thus $|N(U)|_U - |U| < 0$ and there is
a violation of the capped Hall condition.
\end{proof}

If we want a nowhere-zero 1-flow, we must amend the expansiveness condition as well.
Clearly we need at least the capped Hall condition and the previous
equality condition, namely that $|N(U)|=|U|$ implies $|N(N(U))| = |U|$.
But we need not forbid $|N(U)|_U = |U| < |N(N(U))|$; we can allow edges from $X \setminus U$
to $y \in Y$ provided $|y|_U < |y|$.  Accordingly, we say that a set $U \subset X$ is
{\em expansive} in the capacitated case if either $|N(U)|_U > |U|$, or $|N(U)|_U = |U|$
and no edge may connect $X \setminus U$ to $y \in N(U)$ unless $|y|_U < |y|$.  We call $G$
expansive if every $U \subset X$ is expansive, and then we have:

\begin{theorem}\label{thm:capped} A finite, node-weighted, edge-capped bipartite
graph $G$, with positive weights summing to the same value on each bipart, has a full
nowhere-zero flow iff it is expansive.
\end{theorem}

\begin{proof} That the (capped) expansiveness condition is necessary follows from previous
arguments, so let us suppose that $G$ is expansive with the intent of finding a full,
nowhere-zero flow that obeys the edge-capacity constraints.

First let us observe as before that if for every edge $e$ there is a full flow $\phi_e$ on
$G$ with $\phi_e(e) > 0$, then there is indeed a full nowhere-zero flow;
$\phi := \sum_{k=1}^{\infty} 2^{-k} \phi_{e(k)}$ will do the trick, for any numbering
$E = \{e(1), e(2), \dots \}$ of the edges of $G$.  To get $\phi_e$, we would like to proceed
as before, that is, find an $\eps > 0$ so that $G - \eps e$ satisfies the Hall condition,
then increment the weight of $e$ in a flow on $G - \eps e$ to get $\phi_e$.

Subtracting $\eps e$ shaves $\eps$ not only from the weights $|x|$ and $|y|$ of $e$'s
endpoints, but also from $e$'s capacity $c(e)$.  The loss of capacity cannot directly reduce
a set's slack since it comes with a like loss from $|x|$.  But a set $U \subset X$
which does not contain $x$, but whose neighborhood contains $y$, might now flunk the
capped Hall condition as a result of the reduction of $|y|$.  We say that $U$ is
{\em vulnerable} if $x \not\in U$ but $y \in N(U)$.

If $U$ is vulnerable and $|y|_U = |y|$, $|N(U)|_U$ will drop by $\eps$; otherwise, only
by $\max(0, \eps-(|y|-|y|_U))$.  From this observation comes a new slack definition that
depends on $y$:
$$
r(U) = |N(U)|_U - |U| + (|y|-|y|_U) = |y| - |U| + \sum_{v \in N(U) \setminus \{y\}} |v|_U.
$$
The point is that $G - \eps e$ satisfies the capped Hall condition iff $\eps \le r(U)$
for every $U \subset X$.  Thus, the theorem is proved if we can show that there
is no sequence $U_1, U_2, \dots$ of vulnerable sets with $r(U_i) \to 0$.

It is easily checked that for any $v$, $|v|_U$ is submodular as a function of $U$,
hence $r$ is a submodular function, and in particular $r(U \cup V) \le r(U) + r(V)$.
We may assume that $r(U_i) < 2^{-i}$, hence $r(W_n) < 2^{-n}$ where $W_n =
\bigcup_{i>n}U_i$.  Let $W = \bigcap_n W_n$.

Since $N(W) \subseteq \bigcap_n N(W_n)$, we have
$$
r(W) \le \lim_{n \to \infty} r(W_n) = 0.
$$
Since $r(W)$ cannot be negative, we deduce that $N(W) = \bigcap_n N(W_n)$ and thus
$y \in N(W)$; moreover, $|y|_W = |y|$ since $r(W_n) \to 0$ implies $|y| - |y|_{W_n} \to 0$.
Since $x \not\in W$, this means exactly that $W$ is not expansive in the capped sense,
and the proof is complete.
\end{proof}

\subsection{Continuous Marginals}

We begin by noting (reluctantly) that there are closed sets $S \subseteq X \times Y$
which are compliant and expansive but do not support a permuton.  We give three examples
below, each with a different sort of obstruction.

It is convenient to consider what happens on a square ``cell" $C = I \times J$, where
$I$ and $J$ are intervals of $X$ and $Y$, respectively.  (For our purposes it suffices
to consider small cells with rational, or even dyadic, corners.) If $C$ is ``occupied," that is, its interior
intersects $S$, then a permuton supported by $S$ must of course put positive mass on $C$.
But the amount of such mass is limited by its local marginals (on $C$) needing to be Lebesgue-bounded.

Accordingly, we define the {\em capacity} $c(C)$ of an occupied cell $C$ to be the infimum over
measurable sets $A \subseteq I$ of $|L(A)| + |I \setminus A|$, where $L(A)$ is the local
neighborhood $N(A) \cap J$ and $| \cdot |$ denotes Lebesgue measure.  Applying Theorem~\ref{thm:cont-mfmc}
to $C$, with Lebesgue bounds on the marginals, we have:

\begin{lemma}\label{lem:cell-capacity}
No permuton supported in $S$ can put mass greater than $c(C)$ on $C$.  However, there is
a distribution on $C$ of mass exactly $c(C)$ whose local marginals are Lebesgue-bounded.
\end{lemma}

It follows that if there is an occupied cell of zero capacity, $S$ cannot be the support
of a permuton.  Fig.~\ref{fig:horiz} boasts a cell that intersects $S$ in a closed horizontal
line segment; since the local neighborhood of $I$ itself is a single point, this cell has
zero capacity.  But it is easily checked that $S$ is compliant and expansive.

\begin{figure}[ht!]
        \centering
        \includegraphics[width=.8\linewidth,trim=0cm 4cm 0cm 1cm, clip=true]{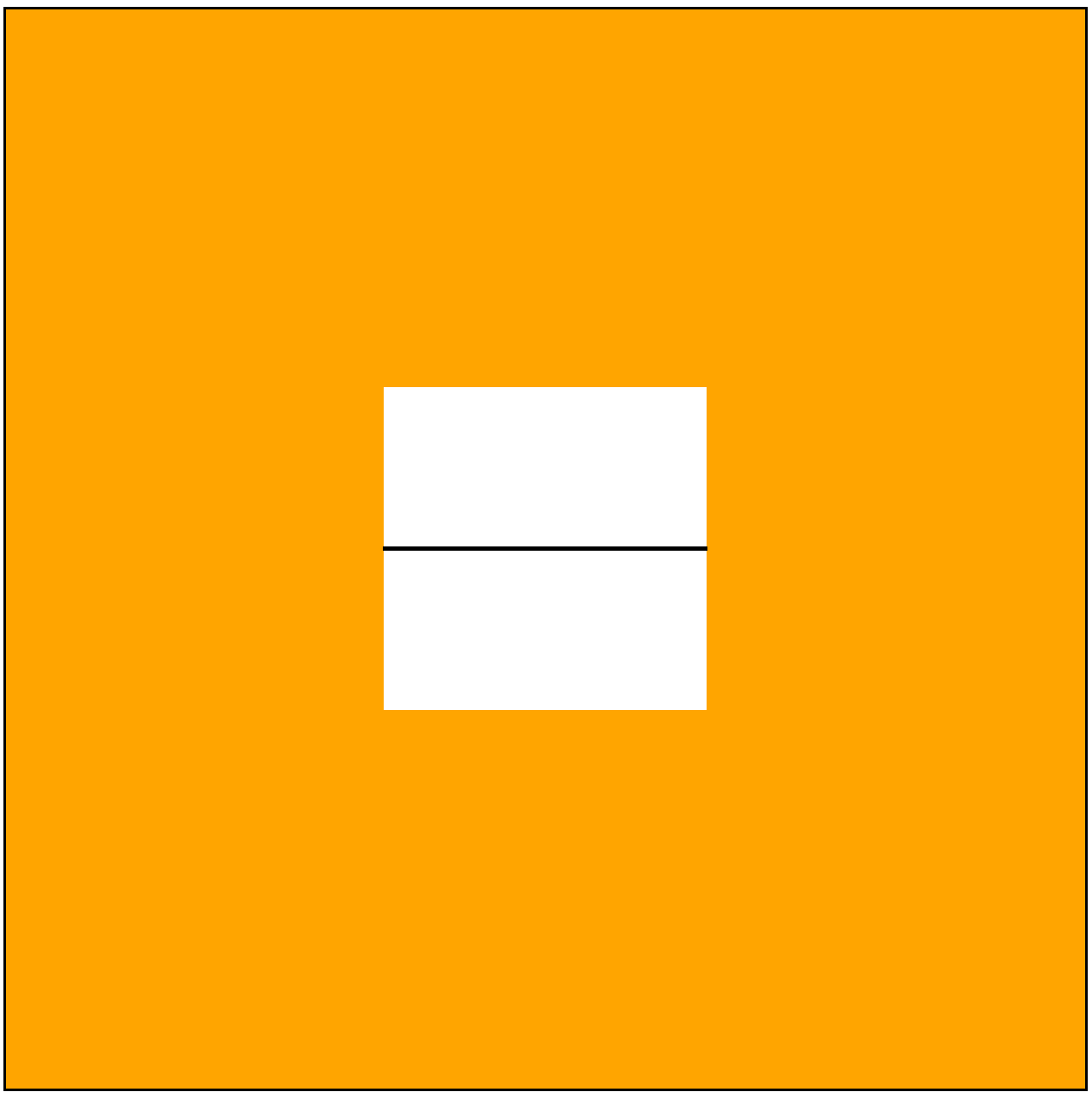}
        \caption{A compliant and expansive set that does not support a permuton.}
        \label{fig:horiz}
\end{figure}

Even if $|L(I)| = |J|$ and $|L(J)| = |I|$, as would be the case if $S \cap C$ is the graph
of a continuous, monotone, one-to-one function from $I$ to $J$ (as suggested by Mateusz Kwa\'snicki
\cite{Kw} in response to the authors' MathOverflow query), we may have $c(C)=0$. An example: Define the $\tau$-measure
of an interval to be the probability that a random real in [0,1], generated by i.i.d.\ selections of
its binary expansion coordinates with $\Pr(1) = 1/3$, lies in the interval.  Take $C = [0,1]^2$ and
take $S$ to be the graph of the cumulative distribution of the singular measure
$\tau$ thus generated.  Let $W$ be the set of all reals in the unit interval whose binary
expansions have limiting density 1/2 of 1's.  Then by the Law of Large Numbers, the Lebesgue
measure of $W$ is 1 but its $\tau$-measure, which is equal to $|L(W)|$, is zero.

Even supposing all occupied cells have positive capacity, we still need the continuous analogs
of the conditions of Theorem~\ref{thm:capped}.  That even they are not enough can be seen
from the following example, suggested by ChatGPT 5.6 and illustrated in Fig.~\ref{fig:farey}.

\begin{figure}[ht!]
        \centering
        \includegraphics[width=.8\linewidth,trim=0cm 5cm 0cm 0cm, clip=true]{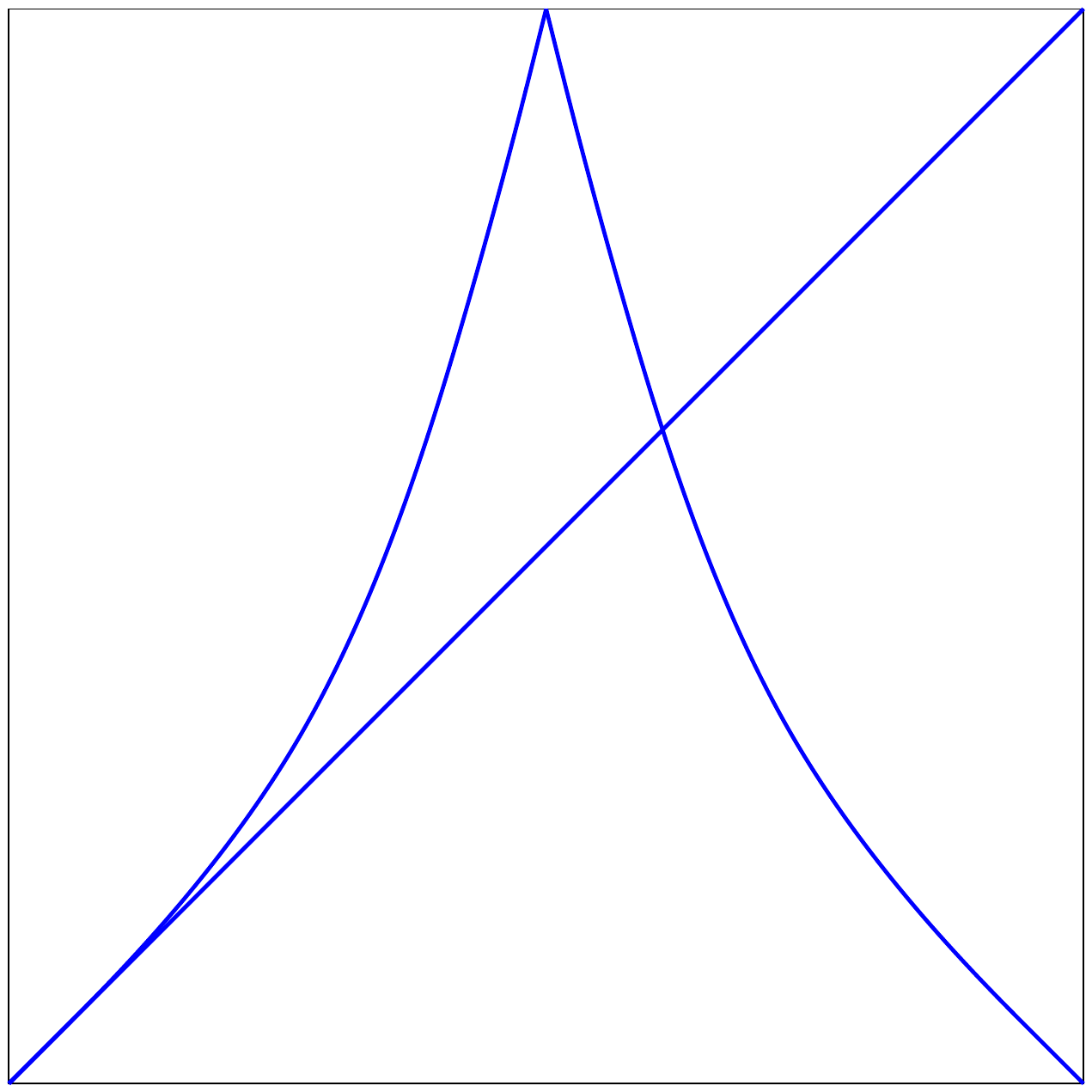}
        \caption{A set that satisfies the continuous analogs of the conditions of Theorem~\ref{thm:capped}
but does not support a permuton.}
        \label{fig:farey}
\end{figure}

The curves in the figure are the graphs of $f(x) = x/(1{-}x)$ on $[0,\frac12]$ and $g(x) = (1{-}x)/x$
on $[\frac12,1]$, but all that matters is that they are strictly convex and asymptotic with absolute
slope 1 at the corners.  That this $S$ is compliant (and capped-compliant) is obvious since the
``identity" permuton (uniform on the main diagonal, 0 elsewhere) is supported in $S$.  But any
attempt to find a permuton supported in $S$ that has mass on one of the curves is doomed; such
a mass will contribute to the $X$-marginal at a region nearer the center than it does to the $Y$-marginal,
but of course any remaining distribution on the main diagonal contributes identically to both marginals.

\medskip

The next theorem provides a characterization (of sorts), which, though neither deep nor satisfying,
will prove useful in the succeeding section.

An occupied cell $C$ will be said to be {\em shavable} if there is a positive assignment of mass on $C$
whose marginals, when ``shaved off" of the uniform marginals of the unit square, leave a compliant
system.  More precisely: $C$ is shavable if there is a nontrivial distribution $\rho$ on $S \cap C$,
with sub-Lebesgue $X$- and $Y$-marginals $\alpha$ and $\beta$, such that the system consisting of $S$
on the unit square with marginals $\lambda - \alpha$ and $\lambda - \beta$ is compliant.

\begin{theorem}\label{thm:cont} Let $S$ be a closed set on the unit square $X \times Y$.  Then $S$ is the
support of a permuton iff $S$ is compliant and every (small, rational) occupied cell is shavable.
\end{theorem}

\begin{proof} If $\gamma$ is a permuton supported by $S$, and $C$ is an occupied cell, then taking
$\rho$ to be $\gamma$ restricted to $S \cap C$ and letting $\gamma' = \gamma - \rho$ leaves a
distribution supported in $S$ with the appropriate shaved marginals, thus a compliant system.

Conversely, let $P$ be any rational point of $S$ and $\eps > 0$.  Let $C$ be a rational cell of
size less than $\eps \times \eps$ that contains $P$ in its interior, and choose a nontrivial distribution
$\rho$ on $S \cap C$ in accordance with the shavability of $C$.  Compliance of the system with $\rho$'s
marginals shaved gives us (by Corollary~\ref{cor:shaved}) a subpermuton $\gamma'$ supported in $S$.  Adding
$\rho$ to $\gamma'$ gives a permuton $\gamma$ still supported in $S$, but with positive mass on $C$.

The usual argument finishes the proof:  Repeating for every rational $P_i \in S$ and $\eps_i = 2^{-i}$
gives a countable list $\gamma_1, \gamma_2, \dots$ of permutons.  Their weighted mean $\sum 2^{-i}\gamma_i$
is supported in $S$ and has positive mass arbitrarily close to every rational point in the closed
set $S$, thus is supported exactly by $S$. \end{proof}

Notice that shavability is (strictly) stronger than expansiveness, since if $A \subset X$ is Hall-tight,
an occupied cell in the forbidden region $N(N(A)) \setminus A$ will not be shavable.  Shavability
can be thought of as expansiveness for both tight and {\em asymptotically} tight sets.

\subsection{The Regular Case}

Supports of permutons can be pretty wild (viz., the ``Brownian permutons"
of \cite{BBFGP,Bo}).  One useful way to generate permutons is by rotating a probability
distribution $\pi$ on $X = [0,1]$ to form what we call (by analogy with circulant matrices)
the {\em circulon} $\gamma_\pi$, as follows.  Let $\gamma_\pi([0,x] \times [0,y]) =
\int_0^y \pi([-t,x-t]) dt$ where the endpoints of $[-t,x-t]$ are understood modulo 1,
so that $[-t,x-t] = [1-t,1] \cup [0,x-t]$ for $x \ge t$ and $[-t,x-t] = [1-t,1-t+x]$ otherwise.

The support of $\gamma_\pi$ is the union of southwest-to-northeast diagonal lines.
For example, if $\kappa$ is the Cantor distribution on [0,1], a singular probability
distribution with no point masses, $\gamma_\kappa$ is a singular permuton whose
support is the union of continuum many diagonals, each of which has zero mass.
On the other hand, suppose we enumerate the rationals $\mathbb{Q} \cap [0,1] =
\{r_1,r_2,\dots\}$ in the unit interval and let $D = [0,1] \setminus
\bigcup_i (r_i,r_i + 2^{-i-1})$.  Then $D$ is a closed set of positive measure
$d < 1$ with no interior.  Let $\pi(A) = \frac1d \lambda(A \cap D)$; then $\pi$
is a nonsingular probability measure on $[0,1]$ and $\gamma_\pi$ a nonsingular
permuton whose support has no interior.

Most nonsingular permutons that have arisen in the literature, and indeed all
permutons of finite entropy (see, e.g., \cite{KKRW}), are supported by
{\em regular} closed sets (a set is regular if it is contained in the closure
of its interior, thus a regular closed set is equal to the closure of its interior).
We show next that for such sets $S$, compliance plus expansiveness is sufficient
(as well as necessary) to support a permuton.

Our characterization, the natural extension of the discrete-case characterizations above,
is often easy to check directly from a geometric description of $S$.

\begin{theorem}\label{thm:regular}
Let $S$ be a regular closed set in the unit square. Then $S$ is the support of a permuton
if and only if it is compliant and expansive.
\end{theorem}

\begin{proof} Applying Theorem~\ref{thm:cont}, we need only show that every occupied
cell $C = I \times J$ is shavable.  Moreover, we can take $C \subset S$, since it
suffices to get mass near any point $P$ in the interior of $S$, and by regularity,
we can find an aligned square around such a $P$ that lies entirely within $S$.

The constraint on mass assigned to such a cell (apart from being bounded by $|I|$) is
due entirely to external issues---namely, subsets of $X \setminus I$ (or, respectively,
of $Y \setminus J$) whose slack might become negative if we shave off too much from the
$Y$- (respectively, $X$-) marginal.  It turns out that in this situation we do more
than prove shavability; we can determine precisely the maximum amount that can be shaved.

In the discrete case $I$ and $J$ were merely points ($x$ and $y$) and the quantity
shaved off just a positive real value ($\eps$).  Here, we need to shave off nontrivial
measures $\alpha$ and $\beta$ on $I$ and $J$, respectively, whose removal from the
uniform marginals on $X$ and $Y$ does not cause the slack of any set $U \subseteq X$ to
drop below zero.  The existence of a measure on $C = C \cap S$ with these
marginals is trivial, as we can just take $\rho$ to be the product measure, that is,
the joint distribution arising from independent coordinates.

Suppose $U \subseteq X \setminus I$.  After the removal of $\beta$ from $\lambda$ (Lebesgue
measure on $Y$), and $\alpha$ from $\lambda$ on $X$, $|U|$ will be unaffected but $|N(U)|$ will
drop by the amount $\beta(N(U) \cap J)$.  So for any subset $B \subseteq J$,
we must somehow ensure that $\beta(B) \le s(U)$ for any $U \subseteq X \setminus I$
whose neighborhood contains $B$.

What about sets $U \subseteq X$ that intersect $I$?  Then $N(U) \supseteq J$, thus
$(Y \setminus N(U)) \cap J = \emptyset$, and $Y \setminus N(U)$ will be dealt with when
we construct $\alpha$.

The key concept below is a function of subsets of $J$, which we call {\em torque}.
For measurable $B \subseteq J$ we define the torque $t(B)$ by
$$
t(B) := \inf_{U \subseteq X \setminus I}\tau(B,U)
$$
where
$$
\tau(B,U) = s(U) + |B \setminus N(U)| = |N(U)| - |U| + |B \setminus N(U)| = |N(U) \cup B| - |U|.
$$

\begin{lemma}\label{lemma:submod}
Torque is a nonnegative, increasing, submodular function, such that for all
$B \subseteq J$,  $t(B) \le |B|$, and $t(B) > 0$ iff $|B|>0$.
\end{lemma}

\begin{proof}
We use the last line in the above display:
$$
t(B) = \inf_{U \subseteq X \setminus I}{|N(U) \cup B| - |U|}.
$$
We note first that the infimum is realized, that is, for each $B$
there is a $U$ such that $t(B) = |N(U) \cup B| - |U|$.  The argument
at the beginning of the proof of Theorem~\ref{thm:cont-mfmc} works as it does
there, since $|N(U) \cup B| - |U|$, being a submodular function
(of $U$) minus a modular function, is submodular in $U$.

That $t$ is nonnegative, Lebesgue-bounded  and increasing (that is, %t(B) \le t(B')$
when $B \subset B'$) is trivial, likewise that $t(B) = 0$ when $|B|=0$.
For the converse of the latter,
if $t(B)=0$ but $|B|>0$, then there is a realizer $U \subseteq X \setminus I$ 
with $|U| \ge |B| >0$ with slack 0. But then, by expansiveness of $S$,
$|N(N(U))| = |U|$ which is impossible because $N(N(U)) \supset I$.

A bit more is needed to get submodularity of $t(B)$ as a function of $B$.
Let $B, B' \subseteq J$,
with $t(B)$ realized by $U$ and $t(B')$ by $U'$.  We will use $U \cap U'$
as a witness for $B \cap B'$ and $U \cup U'$ for $B \cup B'$.  Since
the $-|U|$ terms cancel, it suffices to show that
$$
|N(U \cap U') \cup (B \cap B')| + |N(U \cup U') \cup (B \cup B') \le |N(U) \cup B| 
+ |N(U') \cup B'|.
$$
To do this we classify elements of $Y$ according to whether they belong to
$B \setminus B'$, $B \cap B'$, $B' \setminus B$, or none of them; and
according to whether they belong to $N(U) \setminus N(U')$, $N(U \cap U')$,
$N(U') \setminus N(U)$, or none of them.  It is then routine to check that
an element of $Y$ in any of these 16 categories appears in at least as many
of the two terms on the right-hand side of the inequality as on the left.
\end{proof}

\begin{lemma}\label{lemma:measure}
There is a measure $\beta$ on $J$ of total mass $t(J)$ that is bounded by $t$,
such that $\beta(B) > 0$ when $t(B) > 0$.
\end{lemma}

\begin{proof}
Let $J = [a,b]$; then, since $t$ is increasing and (because it is submodular)
subadditive, $t([a,\cdot])$ is the cumulative distribution of a measure
$\beta_a$ on $J$ of total mass $t(J)$ which is bounded by $t$ and thus absolutely
continuous with respect to Lebesgue measure. 

To ensure the measure of non-Lebesgue-null sets is positive, let $x \in J$ and
let $\beta_x$ be obtained by rotating $\beta_a$, that is, by setting $\beta_x([x,v])
= t([x,v])$ for $v > x$ and $\beta_x([u,x]) = t(J) - t([x,b] \cup [a,u])$ for $u < x$. 
Then we set $\beta(B) = \int_a^b \beta_x(B) dx/(b-a)$. \end{proof}

It is perhaps worth noting that even if $t$ were not increasing, submodularity is
enough to guarantee a measure underneath $t$ of mass $t(J)$.  The discrete analog of
this statement follows from the ``sandwich lemma" of Andr\'as Frank \cite{Fr},
which says that if a submodular set function dominates a supermodular one, then
there is a modular function between them.  Here, we could just discretize with
dyadic intervals, define the supermodular function to agree with $t$ on $J$
but otherwise 0, and un-discretize with the help of compactness.

Having defined the measure $\beta$ on $J$, we now repeat everything with the roles
of $X$ and $Y$, and of $I$ and $J$, exchanged. The result is a measure $\alpha$ on $I$
which, it turns out, has the same total mass.

\begin{lemma}\label{lemma:tItJ} $t(I) = t(J)$.
\end{lemma}

\begin{proof}
We show $t(I) \le t(J)$, and conclude equality by symmetry.  Suppose
$U \subseteq X \setminus I$ witnesses $t(J)$, and let $U' = Y \setminus N(U)$,
so that $s(U')=s(U)$.

We first tackle the (easier) case where $N(U) \supseteq J$, thus $t(J) = s(U)$.
We have $U' \subseteq Y \setminus J$ and thus $U'$ is a valid witness for $t(I)$.
If $U$ is ``full," meaning $N(U') = X \setminus U$, then we have $N(U') \supseteq I$
hence $t(I) \le s(U') = (1 - |U|) - (1 - |N(U)|) = s(U) = t(J)$.

To the extent that $U$ is not full, $U'$'s slack falls short of $U$'s:
$$
s(U') = |N(U')| - |U'| \le 1 - |U| - |I \setminus N(U')| - (1 - |N(U)|)
$$
$$
= |N(U)| - |U| - |I \setminus N(U')| = t(J) - |I \setminus N(U')|.
$$

Thus $t(I) \le \tau(I,U') \le s(U') + |I \setminus N(U')| \le t(J)$ and we
are done with this case, not having used that $|I|=|J|$.

Now suppose that $N(U)$ does not cover $J$, thus $t(J) = \tau(J,U) = s(U) +
|J \setminus N(U)|$.  Then $U' = Y \setminus N(U)$ intersects $J$, thus is not
eligible to be a witness for $t(I)$.  Instead we use $U'' := U' \setminus J$.

Then
$$
t(I) \le \tau(I,U'') = s(U'') + |I \setminus N(U'')| = |N(U'')| - |U''|
+ |I \setminus N(U'')|
$$
$$
= |I| + |N(U'') \setminus I| - (1 - |N(U)| - |J \setminus N(U)|).
$$

In comparing this quantity with $t(J) = |N(U)| - |U| + |J \setminus N(U)|$, the
$|J \setminus N(U)|$ and $|N(U)|$ terms cancel, so we only need to show
$|I| - 1 + |N(U'') \setminus I| \le -|U|$, i.e., $|I| + |U| + |N(U'') \setminus I|
\le 1$.  But $I$, $U$, and $N(U'') \setminus I$ are disjoint subsets of $X$, and
the conclusion follows.
\end{proof}

We have now essentially completed the proof of Theorem~\ref{thm:regular}.
Define a nonsingular measure $\rho$ on the cell $C = I \times J$ ``independently"
from $\alpha$ and $\beta$, by
$$
\rho(A \times B) = \alpha(A) \beta(B)/t(J).
$$
Shaving off the entire product measure $\alpha \times \beta/t(J)$ on $C = C \cap S$ 
removes $\alpha$ and $\beta$ from the uniform $X$- and $Y$-marginals but because
$\alpha$ and $\beta$ are bounded by torque, no subset $U$ of $X$ or $V$ of $Y$ becomes
noncompliant (although some do become tight).

What if $U$ intersects $I$?  Then $Y \setminus N(U)$ is disjoint from $J$, and a symmetric
argument shows $Y \setminus N(U)$, and therefore also $U$ itself, is compliant in
the subpermuton setting.  \end{proof}

\section{Concluding Remarks}

The above theorems together with some simple observations suffice to
settle the case for many sets $S$, including those of Fig.~\ref{fig:test-sets},
repeated as Fig.~\ref{fig:answer-sets} below.

The first of these (in clockwise order from the top left) supports many permutons,
including one with mass distributed uniformly on $S$; the second supports the
permuton obtained from a uniform distribution on a sphere, projected onto its
shadow.  The third fails because (e.g.) the set $[0,\frac14] \cup [\frac34,1] \subset X$
is not expansive; the fourth fails because $[.49,.51]$ violates the
Hall condition.  In the fifth (lower right of Fig.~\ref{fig:test-sets}),
$(\frac14,1]$ flunks compliance.
In the sixth, the steepness of the curves causes $[.1,.9]$ to be noncompliant;
but the better-controlled derivatives in the seventh allow mass split between
the two curves with uniform marginals. Finally, in the last case, the dangerous
set is $[\frac23,1]$ but its neighborhood $[\frac35,1]$ gives it positive slack.

\begin{figure}[ht!]
        \centering
        \includegraphics[width=.8\linewidth,trim=0cm 9cm 0cm 6cm, clip=true]{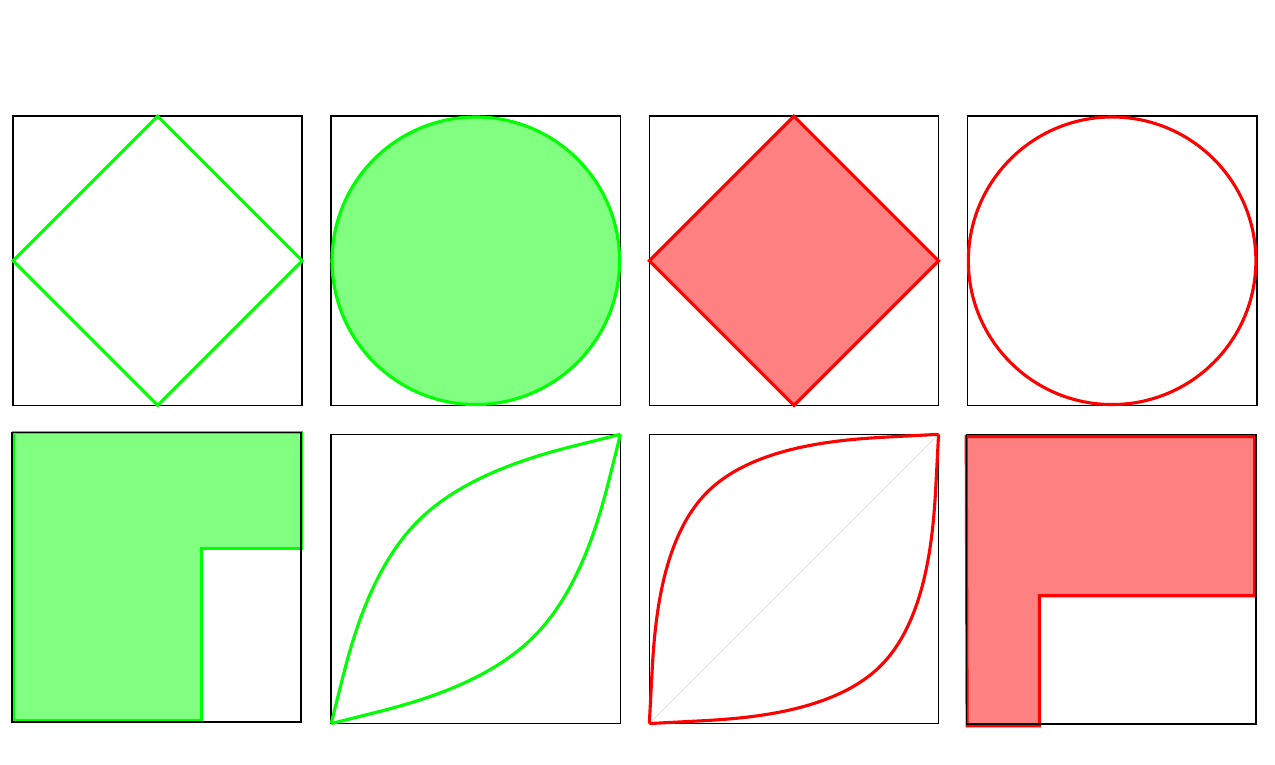}
        \caption{The four subsets of $[0,1]^2$ on the left side support a permuton,
        the rest do not.}
        \label{fig:answer-sets}
\end{figure}


\begin{thebibliography}{1234}

\bibitem {BBFGP} F.\ Bassino, M.\ Bouvel, V.\ F\'eray, L.\ Gerin, and A.\ Pierrot,
The Brownian limit of separable permutations, {\em Ann.\ Prob.} {\bf 46} \#4 (2018), 2134--2189.

\bibitem{Bo} J.\ Borga, The skew Brownian permuton: A new universality class for random constrained permutations,
{\em Proc.\ London Math.\ Soc.} {\bf 126} \#6 (June 2023), 1842--1883.

\bibitem{BDMW} J.\ Borga, S.\ Das, S.\ Mukherjee, and P.\ Winkler, Large deviation principle for random
permutations, {\em International Math.\ Res.\ Notices}, Volume 2024, Issue 3 (February 2024), 2138--2191.
\verb+https://doi.org/10.1093/imrn/rnad096+.

\bibitem{C} V.\ Chv\'atal, {\em Linear Programming}, W.H. Freeman, 1983.

\bibitem{DK} P.\ Devlin and J.\ Kahn, Perfect fractional matchings in k-out hypergraphs,
{\em Electronic Journal of Combinatorics} {\bf 24} \#3 (2017), paper \#P3.60.

\bibitem{Fr} A.\ Frank, An algorithm for submodular functions on graphs, {\em Ann.\ 
Disc.\ Math.} {\bf 16} (1982), 97--120.

\bibitem{HKMS} C.\ Hoppen, Y.\ Kohayakawa, C.G.\ Moreira, B.\ R\'{a}th, and R.M.\ Sampaio, Limits of permutation
sequences, {\em J.\ Combin.\ Theory B} {\bf 103} (2013) 93--113. 

\bibitem{KKRW} R.\ Kenyon, D. Kr\'{a}l', C. Radin, and P. Winkler, Permutations with fixed pattern densities,
{\em Random Structures \& Algorithms} {\bf 56} \#2 (2020) (online 31 July 2019), 220--250.  doi:10.1002/rsa.20882.

\bibitem{Kw} M.\ Kwa\'snicki (\verb+https://mathoverflow.net/users/108637+),
A possible measure-theoretic pathology, URL (version: 2023-09-22): \verb+https://mathoverflow.net/q/454673+

\bibitem{L2} L.\ Lov\'{a}sz, Submodular functions and convexity. In: Bachem, A., Korte, B., Grötschel, M.
(eds) {\em Mathematical Programming: The State of the Art}. Springer, Berlin, Heidelberg (1983), 235--257.
\verb+https://doi.org/10.1007/978-3-642-68874-4_10+

\bibitem{Mu} S.\ Mukherjee, Estimation in exponential families on
  permutations, {\em Ann.\ Stat.} {\bf 44} \#2 (April 2016), 853--875.

\bibitem{Sklar} A.\ Sklar, Fonctions de répartition à $n$ dimensions et leurs marges,
{\em Publ.\ Inst.\ Statist.\ Univ.\ Paris} {\bf 8} (1959), 229--231.

\bibitem{Tr} J.\ Trashorras, Large deviations for symmetrised
  empirical measures, {\em J.\ Theor.\ Probab.} {\bf 21} (2008), 397--412.

\end{thebibliography}
\end{document}